\documentclass{article}
\usepackage[utf8]{inputenc}
\usepackage[T1]{fontenc}

\usepackage{listings}
\usepackage{xcolor}
\usepackage{hyperref}
\usepackage{graphicx}
\graphicspath{{img/}}

\usepackage{tabularx}
\usepackage{multirow}
\usepackage{makecell}
\usepackage{booktabs}
\usepackage[labelfont=bf,format=plain,justification=raggedright,singlelinecheck=false]{caption}

\usepackage{amsfonts}
\usepackage{amsmath}
\usepackage{amsthm}

\usepackage{siunitx}
\usepackage[capitalise]{cleveref}

\newtheorem{theorem}{Theorem}
\newtheorem{lemma}{Lemma}

\newtheorem{remark}{Remark}

\definecolor{codegreen}{rgb}{0,0.6,0}
\definecolor{codegray}{rgb}{0.5,0.5,0.5}
\definecolor{codepurple}{rgb}{0.58,0,0.82}
\definecolor{backcolour}{rgb}{0.95,0.95,0.92}

\lstdefinestyle{mystyle}{
    backgroundcolor=\color{backcolour},   
    commentstyle=\color{codegreen},
    keywordstyle=\color{magenta},
    numberstyle=\tiny\color{codegray},
    stringstyle=\color{codepurple},
    basicstyle=\ttfamily\footnotesize,
    breakatwhitespace=false,         
    breaklines=true,                 
    captionpos=b,                    
    keepspaces=true,                 
    numbers=left,                    
    numbersep=5pt,                  
    showspaces=false,                
    showstringspaces=false,
    showtabs=false,                  
    tabsize=2
}
\newcommand{\Complex}{\mathbb{C}}
\newcommand{\Real}{\mathbb{R}}
\newcommand{\NN}{\mathbb{N}}
\newcommand{\conj}[1]{\overline{#1}}
\newcommand{\re}[1]{\ensuremath{\textmd{re}}(#1)}
\newcommand{\im}[1]{\ensuremath{\textmd{im}}(#1)}

\newcommand{\float}[1]{\ensuremath{\hat{#1}}}
\newcommand{\sign}[1]{\ensuremath{\textmd{sign}}(#1)}

\title{Numerical analysis of Givens rotation}
\author{Weslley S Pereira  \and Ali Lotfi \and Julien Langou}
\date{%
    {\footnotesize Department of Mathematical and Statistical Sciences, University of Colorado Denver}\\[2ex]%
    \today
}

\begin{document}
\maketitle

\begin{abstract}
    Generating 2-by-2 unitary matrices in floating-precision arithmetic is a delicate task. One way to reduce the accumulation error is to use less floating-point operations to compute each of the entries in the 2-by-2 unitary matrix. This paper shows an algorithm that reduces the number of operations to compute the entries of a Givens rotation. Overall, the new algorithm has more operations in total when compared to algorithms in different releases of LAPACK, but less operations per entry. Numerical tests show that the new algorithm is more accurate on average.
\end{abstract}

\section{Introduction}

A Givens rotation $Q \in \mathbb{K}^{2\times 2}$ associated to the pair $(f,g) \in \mathbb{K}^2$, $\mathbb{K} = \Real$ or $\Complex$, is a unitary matrix
\begin{align*}
	Q = \begin{bmatrix}
		c & s\\
		-\conj{s} & c
	\end{bmatrix},
\end{align*}
that satisfies
\begin{align*}
	\begin{bmatrix}
		c & s\\
		-\conj{s} & c
	\end{bmatrix} \begin{bmatrix}
		f\\
		g
	\end{bmatrix}
	= \begin{bmatrix}
		r\\
		0
	\end{bmatrix}
\end{align*}
for a particular $r \in \mathbb{K}$. In the real case, $c$ and $s$ represent sine and cosine, respectively, and, with this information, one may compute the rotation angle. Notice that the Givens rotation is not unique since $-Q$ is also a Givens rotation for every Givens rotation $Q$. Moreover, there are infinite possible matrices $Q$ if $f = 0$ and $g \in \Complex$ \cite{Demmel2002}.

There are a couple of papers dedicated to the Givens rotation algorithms in the BLAS and the LAPACK library. The real-arithmetic SROTG and DROTG were first presented in the first-level BLAS documentation \cite{Lawson1979}, and then the LAPACK routines SLARTG, DLARTG, CLARTG and ZLARTG appeared in \cite{Anderson1999}. These algorithms were all reviewed in \cite{Demmel2002}, where the authors propose a new algorithm for the complex-arithmetic case which improves performance and accuracy. This algorithm uses a single square root and a single division, and it is very similar to the algorithm implemented in LAPACK 3.10. The complex-arithmetic Givens rotation algorithm in LAPACK 3.10 was presented in \cite{Anderson2017} and is a slight variation of \cite{Demmel2002} with respect to the scaling of quantities.

The present work is motivated by a bug report about the algorithm in LAPACK 3.10 (See \href{https://github.com/Reference-LAPACK/lapack/issues/629}{github.com/Reference-LAPACK/lapack/issues/629}). The author reported that the new Givens rotations may have lower accuracy than the ones that were in LAPACK up to release 3.9. This could be easily verified by noticing that, after applying several rotations to a unitary $2 \times 2$ matrix, the departure from unitaricity was larger (in average) for 3.10 than 3.9. We performed additional experiments where we couldn't find the referred problem nor verify that the algorithm in LAPACK 3.10 was less accurate.
Moreover, neither \cite{Demmel2002} nor \cite{Anderson2017} deal with the problem of applying several rotations to a single matrix.

This work complements the numerical analysis on the generation of Givens rotations algorithms from previous works in two ways: (1) worst-case scenario analysis; (2) probabilistic distribution of the error after applying several rotations. 

\subsection{Algorithms in real arithmetic}

If $f,g \in \Real$, the Givens rotation is given by
\begin{align}\label{eq:realGivensFormulas}
    c = p \frac{f}{\sqrt{f^2+g^2}}, \quad
    s = p \frac{g}{\sqrt{f^2+g^2}}, \quad
    r = p \sqrt{f^2+g^2},
\end{align}
where $p = 1$ or $-1$,
which are unique expressions aside from one choice of sign.

When no scaling is necessary and both $f$ and $g$ are non zero, \href{https://github.com/Reference-LAPACK/lapack/blob/lapack-3.9/SRC/slartg.f}{LAPACK 3.9 \texttt{slartg}} computes the Givens rotation as follows:
\begin{lstlisting}[language=Fortran, firstline=153, lastline=154, firstnumber=153]
         F1 = F
         G1 = G
\end{lstlisting}
\vspace{-10pt}
\begin{lstlisting}[language=Fortran, firstline=187, lastline=189, firstnumber=187]
            R = SQRT( F1**2+G1**2 )
            CS = F1 / R
            SN = G1 / R
\end{lstlisting}
\vspace{-10pt}
\begin{lstlisting}[language=Fortran, firstline=191, lastline=195, firstnumber=191]
         IF( ABS( F ).GT.ABS( G ) .AND. CS.LT.ZERO ) THEN
            CS = -CS
            SN = -SN
            R = -R
         END IF
\end{lstlisting}
and
\href{https://github.com/Reference-LAPACK/lapack/blob/lapack-3.10/SRC/slartg.f90}{LAPACK 3.10 \texttt{slartg}} computes the Givens rotation as follows:
\begin{lstlisting}[language=Fortran, firstline=133, lastline=134, firstnumber=153]
   f1 = abs( f )
   g1 = abs( g )
\end{lstlisting}
\vspace{-10pt}
\begin{lstlisting}[language=Fortran, firstline=145, lastline=149, firstnumber=187]
      d = sqrt( f*f + g*g )
      p = one / d
      c = f1*p
      s = g*sign( p, f )
      r = sign( d, f )
\end{lstlisting}

One can see that the two codes behave differently when $|g| \ge |f|$ and $f < 0$. In this region, $p = +1$ in LAPACK 3.9 and $p=-1$ in LAPACK 3.10. The algorithm in LAPACK 3.10 uses $p = \sign{f}$.
The algorithm in \cite[Section~19.6]{Higham2002} is a variant that uses $p=1$.
The real-valued Givens rotation algorithm in LAPACK 3.10 is compatible with its complex-valued algorithm (see the complex case below). This means that the outputs of \texttt{slartg} and \texttt{clartg} approximate the same real quantity for any real pair $(f,g)$. In LAPACK 3.9, \texttt{slartg} and \texttt{clartg} approximate different real quantities when $|f| \le |g|$ and $f < 0$.

Using the simplest algorithm for computing \cref{eq:realGivensFormulas}, we obtain
\begin{align*}
    \float{c} &= \frac{f}{\sqrt{(f^2(1+\delta_1)+g^2(1+\delta_2))(1+\delta_3)}(1+\delta_4)}(1+\delta_5) \\
    &= \frac{f}{\sqrt{f^2+g^2}}\sqrt{\frac{(1+\delta_5)^2}{(1+\delta_6)(1+\delta_3)(1+\delta_4)^2}} \\
    &= c \sqrt{1+\theta_6} \\
    &= c (1+\theta_4),
\end{align*}
where $\delta_6 := (f^2 \delta_1 + g^2 \delta_2) / (f^2+g^2)$, and we use \cref{th:squareRootOnePlusTheta}.
Similar calculations can be used to obtain $\float{s} = s (1+\theta_4')$ and $\float{r} = r (1+\theta_3)$. This is in accord with \cite[Lemma~19.7]{Higham2002}.
The algorithm from LAPACK 3.10 uses $p = 1 / d$, having one additional floating-point operation than the algorithm from LAPACK 3.9. One should expect larger errors in the worst-case scenario using LAPACK 3.10.

\subsection{Algorithms in complex arithmetic}

Let $f,g \in \Complex$. Then, the Givens rotation can be defined as
\begin{align*}
    c &= \frac{|f|}{\sqrt{|f|^2+|g|^2}} \in \Real, \\
    s &= \sign{f}\frac{\conj{g}}{\sqrt{|f|^2+|g|^2}} \in \Complex, \\
    r &= \sign{f}\sqrt{|f|^2+|g|^2} \in \Complex.
\end{align*}
where $|f| := \sqrt{\re{f}^2+\im{f}^2}$ and $\sign{f} := f / |f| \in \Complex$.
When $f = 0$ we may consider: $c = 0$, $s = \sign{\conj{g}}$ and $r = |g|$.
When $g = 0$ we may again consider the ``least work aproach'': $c = 1$, $s = 0$ and $r = f$.
We postpone the details about the algorithms for complex Givens rotations in complex arithmetic to \cref{sec:accuracyComplex}.
All those algorithms approximate $c$, $s$ and $r$ as given above.

\subsection{Outline}

This paper is organized as follows:
\Cref{sec:squareRoot} presents the notation and some results needed for the analysis of the Givens rotation algorithms in complex arithmetic.
The worst-case scenario analysis for the algorithms in LAPACK 3.9 and LAPACK 3.10 is performed in \Cref{sec:accuracyComplex}. In the same section, we introduce the new algorithm that has smaller relative errors.
\Cref{sec:numerical} validates the analysis and compares the accuracy and performance of the algorithms. 
The conclusions are presented in \cref{sec:concl}.


\section{Square root of rounding errors}
\label{sec:squareRoot}

This section uses notation from \cite[Chapter 3]{Higham2002} to obtain estimates for $\sqrt{1+\theta_n}$, which is useful to the analysis of the Givens rotation algorithms.
As in \cite{Higham2002}, the symbol $\delta$ is used to represent an arbitrary real quantity satisfying $|\delta| \le u$ and which can be different on each occurrence. The symbol $u$ represents the unit roundoff.
We use $\delta_i$, $i \in \NN$, whenever we want to emphasize the distinction between different $\delta$.
Moreover, we define
\begin{align}
	\gamma_x := \frac{xu}{1-xu}\,,
\end{align}
where $x \in [0,1/u)$, and $\theta_n$, $n \in \NN$, to every quantity satisfying
\begin{align}
	1+\theta_n := \prod_{i=1}^n (1+\delta_i)^{\rho_i}\,, \qquad \rho_i = \pm 1\,.
\end{align}
We also allow for $\theta_n$ to change on each occurrence.

Starting from $|\theta_n| \le \gamma_n$, we can easily derive
$$0 < 1-\gamma_n \le \sqrt{1-\gamma_n} \le \sqrt{1+\theta_n} \le \sqrt{1+\gamma_n} \le 1+\gamma_n$$
for $nu < 1/2$. The following lemmas give better estimates for 
$\sqrt{1+\gamma_n}$ and $\sqrt{1-\gamma_n}$.

\begin{lemma}\label{th:squareRootOnePlusGamma}
$\sqrt{1+\gamma_x} = 1+\gamma_{\bar{\alpha}x}$ for all $x \in (0,1/u)$, where
\begin{align}\label{eq:squareRootOnePlusGamma}
	\bar{\alpha} = \bar{\alpha}(x) = \frac{1 - \sqrt{1-xu}}{xu}\,.
\end{align}
Moreover, $\bar{\alpha} \in \left(\frac12,1\right)$.
\end{lemma}
\begin{proof}
Let $y := \bar{\alpha}x$ and observe that
\begin{align*}
    (1+\gamma_y)^2 - (1+\gamma_x)
    = \gamma_y(2+\gamma_y)-\gamma_x\,.
\end{align*}
We want to prove the right-hand side is identically zero if we use \cref{eq:squareRootOnePlusGamma}.
In fact,
\begin{align*}
    \gamma_y(2+\gamma_y)-\gamma_x &= \frac{yu}{1-yu}\frac{2-yu}{1-yu} - \frac{xu}{1-xu} = \frac{2yu - xu - (yu)^2}{(1-yu)^2(1-xu)}\,.
\end{align*}
The discriminant of the polynomial $p(yu) := 2(yu) - xu - (yu)^2$ is $\Delta = 4(1-xu)$, which is always positive. This means $p(yu)$ has exactly 2 real roots. Also, $p''(yu) < 0$, which means it assumes positive values between its two roots $yu = 1 \pm \sqrt{1-xu}$. Now, use the definition of $y$ to conclude that \cref{eq:squareRootOnePlusGamma} is the lowest number so that $\sqrt{1+\gamma_x} = 1+\gamma_{\bar{\alpha}x}$.
It is straight-forward to prove that $\frac12 < \frac{1 - \sqrt{1-xu}}{xu} < 1$.
\end{proof}

\begin{lemma}\label{th:squareRootOneMinusGamma}
$\sqrt{1-\gamma_x} = 1-\gamma_{\alpha x}$ for all $x \in (0,1/(2u))$, where
\begin{align}\label{eq:squareRootOneMinusGamma}
	\alpha = \alpha(x) = \frac{1 - \sqrt{1-xu(3-2xu)}}{xu(3-2xu)}\,.
\end{align}
Moreover, $\alpha \in \left(\frac12,1\right)$.
\end{lemma}
\begin{proof}
Let $y := \alpha x$ and observe that
\begin{align*}
    (1-\gamma_y)^2 - (1-\gamma_x)
    = \gamma_y(2-\gamma_y)-\gamma_x\,.
\end{align*}
We want to prove the right-hand side is identically zero if we use \cref{eq:squareRootOneMinusGamma}.
In fact,
\begin{align*}
    \gamma_y(2-\gamma_y)-\gamma_x &= \frac{yu}{1-yu}\frac{2-3yu}{1-yu} - \frac{xu}{1-xu} = \frac{2yu - xu - (yu)^2 (3-2xu)}{(1-yu)^2(1-xu)}\,.
\end{align*}
The discriminant of the polynomial $p(z) = 2z - x - z^2 u (3-2xu)$, $z := yu$, is $\Delta = 4(1-xu(3-2xu))$.
We need to look at the sign of $q(w) = 1 - 3w + 2w^2$, $w := xu$.

The discriminant of $q(w)$ is $\Delta = 9 - 8 = 1$, which means it also has 2 square roots. Since $q''(w) > 0$, we want values $w$ not in between the two roots, i.e., $w \le 1/2$ or $w \ge 1$. Since $w$ is always less than 1, we need to require that $w = xu \le 1/2$.

The rest of the proof follows the same arguments of the proof of \cref{th:squareRootOnePlusGamma}.
\end{proof}

We have estimates for both $\sqrt{1+\gamma_n}$ and $\sqrt{1-\gamma_n}$ for a large range of values of $n$. The next result gives the estimates for $\sqrt{1+\theta_n}$.

\begin{theorem}\label{th:squareRootOnePlusTheta}
Let $n \in \NN$, $\theta \in \Real$ s.t. $nu \le 1/2$ and $1+\theta = \sqrt{1+\theta_n}$. Therefore,
\begin{align*}
|\theta| \le \gamma_{\alpha n},
\end{align*}
where $\alpha$ is given by \cref{eq:squareRootOneMinusGamma}.
Moreover, if $( 3 - 2 n u )(\lfloor\frac{n}2 \rfloor+1)^2 u \le (2 + \lfloor\frac{n}2 \rfloor - \lceil \frac{n}2\rceil)$, then $\theta = \theta_{\lfloor \frac{n}2 \rfloor + 1}$.
\end{theorem}
\begin{proof}
    First, notice that
    \begin{align*}
        (1+\gamma_y)^2 - (1+\gamma_x)
        &= \gamma_y(2+\gamma_y)-\gamma_x\\
        &\ge \gamma_y(2-\gamma_y)-\gamma_x
        = (1-\gamma_y)^2 - (1-\gamma_x)\,,
    \end{align*}
    for all $x,y \in \Real$, which means that $\bar{\alpha} \le \alpha$ from \cref{th:squareRootOnePlusGamma,th:squareRootOneMinusGamma}. Moreover, since $nu \le 1/2$ and $\sqrt{1-\gamma_n} - 1 \le |\sqrt{1+\theta_n} - 1| \le \sqrt{1+\gamma_n} - 1$, \cref{th:squareRootOnePlusGamma,th:squareRootOneMinusGamma} and $\bar{\alpha} \le \alpha$ can be used to conclude that $|\theta| \le \gamma_{\alpha n}$.
    It remains to verify the conditions for $|\theta| \le \gamma_{\lfloor \frac{n}2 \rfloor + 1}$.
    
    Since $nu < 1$, we can use the proof of \cref{th:squareRootOneMinusGamma} to obtain
    \begin{align*}
        &\gamma_{\lfloor \frac{n}2 \rfloor + 1}(2-\gamma_{\lfloor \frac{n}2 \rfloor + 1})-\gamma_n \ge 0\\
        &\quad\Leftrightarrow\quad 2\left(\left\lfloor \frac{n}2 \right\rfloor + 1\right)-n-\left(\left\lfloor \frac{n}2 \right\rfloor + 1\right)^2u(3-2nu) \ge 0\\
        &\quad\Leftrightarrow\quad (3-2nu)\left(\left\lfloor \frac{n}2 \right\rfloor + 1\right)^2u \le 2 + 2\left\lfloor \frac{n}2 \right\rfloor - n\\
        &\quad\Leftrightarrow\quad (3-2nu)\left(\left\lfloor \frac{n}2 \right\rfloor + 1\right)^2u \le 2 + \left\lfloor \frac{n}2 \right\rfloor - \left\lceil \frac{n}2 \right\rceil\,,
    \end{align*}
    which corresponds to the conditions given above.
\end{proof}

\Cref{th:squareRootOnePlusTheta} states that $\sqrt{1+\theta_n}$ is bounded by $1+\gamma_{\alpha n}$, where $\alpha \in (1/2,1)$. Computing $\alpha$ may be tedious and, in most cases, we want $\alpha n \in \NN$. \Cref{th:squareRootOnePlusTheta} provides the bound $1+\gamma_{\lfloor \frac{n}2 \rfloor + 1}$ for $\sqrt{1+\theta_n}$, for small $n \in \NN$. Small here means $n \le \num{4728}$ in single precision and $n \le \num{109588316}$ in double precision.

\begin{remark}\label{th:squareRootSmallN}
    Notice that $\alpha \approx 1/2$ for very small $n$, which means that one should expect $\sqrt{1+\theta_n}$ to be bounded by $1+\gamma_{\frac{n}2}$ for practical purposes.
    \Cref{fig:gammaSmallN} compares the estimates $\gamma_{\frac{n}2}$ and $\gamma_{\lfloor \frac{n}2 \rfloor + 1}$ with $\gamma_{\alpha n}$ for $n = 1, \cdots, 20$ in single precision. The relative difference between $\gamma_{\alpha n}$ and $\gamma_{\frac{n}2}$ is less than $16u$ for those values of $n$.
    For instance, one should expect $|c-\float{c}|/|c|$ and $|s-\float{s}|/|s|$ are closer to $\gamma_3$ in \cite[Lemma~19.7]{Higham2002}. Since $\gamma_3 / \gamma_4  <  0.75$ in single and double precision, the accumulation error after applying several rotations would also be less than $0.75$ of what is predicted by the lemma.
\end{remark}

\begin{figure}[!htb]
\includegraphics[width=\linewidth]{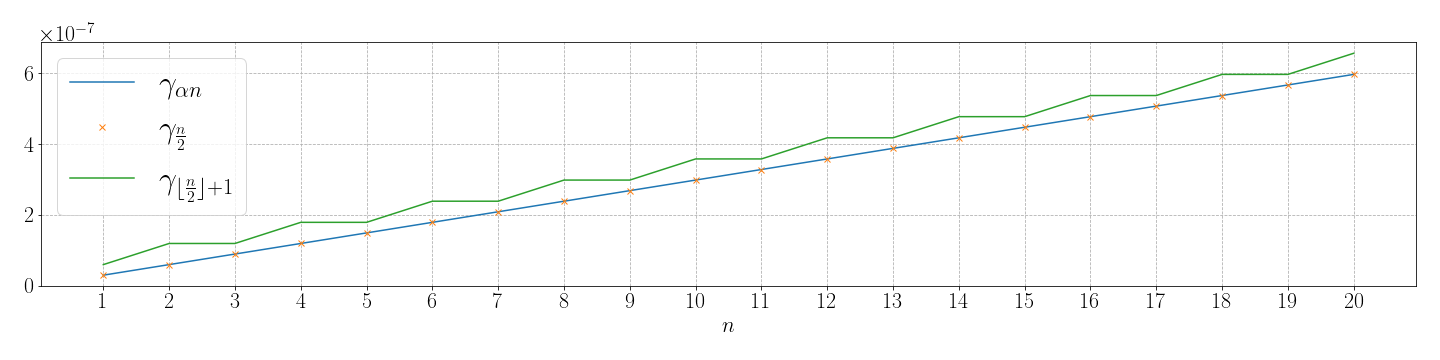}
\caption{Comparison betweeen $\gamma_{\alpha n}$, $\gamma_{\frac{n}2}$ and $\gamma_{\lfloor \frac{n}2 \rfloor + 1}$ for small $n$.}
\label{fig:gammaSmallN}
\end{figure}

\section{Analysis of algorithms in complex arithmetic}
\label{sec:accuracyComplex}

In this section, we compare three algorithms for generating Givens rotations in complex arithmetic: the algorithm in LAPACK 3.9, the algorithm in LAPACK 3.10, and a new proposal. The latter aims to reduce the accumulation errors overall for each of $c$, $s$ and $r$. Hereafter, we use the notation \float{x} that denotes the computed value of $x$.

\subsection{Preliminaries}

In the following, we will be also interested in the accuracy of the rotation matrix
\begin{align}
	\float{Q} = \begin{bmatrix}
		\float{c} & \float{s}\\
		-\conj{\float{s}} & \float{c}
	\end{bmatrix}
\end{align}
in comparison to 
\begin{align}
	Q = \begin{bmatrix}
		c & s\\
		-\conj{s} & c
	\end{bmatrix}.
\end{align}
We assume the floating-point arithmetic is commutative on all operations above.
The accuracy of $\float{Q}$ relies on:
\begin{itemize}
\item Orthogonality of the columns. In this case, $\float{c}\float{s} + (-\float{s}\float{c}) = 0$.
\item Norm of the columns. The error can be measured by $1 - \sqrt{|\float{c}|^2 + |\float{s}|^2}$.
\item Backward error, i.e., $\|\float{Q}^H(\float{r},0)^T - (f,g)^T\|$.
\end{itemize}

Suppose $\float{c} = c(1+\theta_a)$, $\float{s} = s(1+\theta_b)$. Then,
\begin{align*}
\sigma &= \sqrt{|\float{c}|^2 + |\float{s}|^2}
= \sqrt{|c(1+\theta_a)|^2 + |s(1+\theta_b)|^2}\\
&= \sqrt{c^2(1+\theta_a)^2 + |s|^2(1+\theta_b)^2}
\end{align*}
Then, if $m = \max(a,b)$,
\begin{align*}
\sigma &\le 1+\gamma_m \;\Rightarrow\; \sigma-1 \le \gamma_m,\\
\sigma &\ge 1-\gamma_m \;\Rightarrow\; \sigma-1 \le -\gamma_m,
\end{align*}
which means that $\sigma = 1( 1 + \theta_m)$.
As for the backward error, we need to define $\float{r} = r(1+\theta_d)$, and then
\begin{align*}
\float{Q}^H
\begin{bmatrix}
	\float{r}\\
	0
\end{bmatrix}
= \begin{bmatrix}
	\float{c}\float{r}\\
	\conj{\float{s}}\float{r}
\end{bmatrix}
= \begin{bmatrix}
	f\\
	g
\end{bmatrix} (1+\theta_{(m+d)}).
\end{align*}
Thus, the accuracy of $\float{Q}$ relies on how well the algorithm can approximate the pair $(c,s)$, and $r$.

Since $\float{Q}$ is orthogonal, but not necessarily unitary, we can write $\float{Q} = \sigma \tilde{Q}$, where $\tilde{Q}$ is unitary and $\sigma = \sqrt{|\float{c}|^2 + |\float{s}|^2} \in \Real$ is the only singular value of $\float{Q}$. In fact,
\begin{align*}
	\tilde{Q}^H\tilde{Q}
	&= \frac{1}{|\float{c}|^2 + |\float{s}|^2}
	\begin{bmatrix}
		\float{c} & -\float{s}\\
		\conj{\float{s}} & \float{c}
	\end{bmatrix}
	\begin{bmatrix}
		\float{c} & \float{s}\\
		-\conj{\float{s}} & \float{c}
	\end{bmatrix}\\
	&= \frac{1}{|\float{c}|^2 + |\float{s}|^2}
	\begin{bmatrix}
		|\float{c}|^2 + |\float{s}|^2 & \float{c}\float{s} -\float{s}\float{c}\\
		\float{c}\conj{\float{s}} - \conj{\float{s}}\float{c} & |\float{c}|^2 + |\float{s}|^2
	\end{bmatrix} = I_{2\times 2}
\end{align*}
Notice that a product of several finite-precision rotations can be seen as the product of its singular values times a single unitary matrix. This means that $\prod_{i=1}^{m}\float{Q_i} = \left(\prod_{i=1}^{m}\sigma_i\right) \tilde{Q}$ for some unitary matrix $\tilde{Q} \in \Complex^{2\times 2}$.

\subsection{LAPACK 3.9}
\label{sec:analysis391}

When no scaling is necessary and both $f$ and $g$ are non zero, LAPACK 3.9 \href{https://github.com/Reference-LAPACK/lapack/blob/lapack-3.9/SRC/clartg.f}{\texttt{clartg}} computes the Givens rotation as follows:
\begin{lstlisting}[language=Fortran, firstline=152, lastline=153, firstnumber=152]
      FS = F
      GS = G
\end{lstlisting}
\vspace{-10pt}
\begin{lstlisting}[language=Fortran, firstline=178, lastline=179, firstnumber=178]
      F2 = ABSSQ( FS )
      G2 = ABSSQ( GS )
\end{lstlisting}
\vspace{-10pt}
\begin{lstlisting}[language=Fortran, firstline=223, lastline=230, firstnumber=223]
         F2S = SQRT( ONE+G2 / F2 )
*        Do the F2S(real)*FS(complex) multiply with two real multiplies
         R = CMPLX( F2S*REAL( FS ), F2S*AIMAG( FS ) )
         CS = ONE / F2S
         D = F2 + G2
*        Do complex/real division explicitly with two real divisions
         SN = CMPLX( REAL( R ) / D, AIMAG( R ) / D )
         SN = SN*CONJG( GS )
\end{lstlisting}

We now may analyze the approximation error involved on each of the outputs $c$, $s$ and $r$. First, let us look at the approximation of $c$:
\begin{align*}
    f_2 &= (\re{f}^2(1+\delta_1) + \im{f}^2(1+\delta_1'))(1+\delta_3),
    = |f|^2 (1+\delta_1'')(1+\delta_3)\\
    g_2 &= (\re{g}^2(1+\delta_2) + \im{g}^2(1+\delta_2'))(1+\delta_4),
    = |g|^2 (1+\delta_2'')(1+\delta_4)\\
    \float{c} &= \frac{1}{\sqrt{(1+\frac{g_2}{f_2}(1+\delta))(1+\delta)}(1+\delta)}(1+\delta)\\
    &= \frac{(1+\delta)}{\sqrt{(1+\frac{|g|^2}{|f|^2}(1+\theta_5))(1+\delta)}(1+\delta)}\\
    &= c \frac{(1+\delta)}{\sqrt{(1+\theta_5')(1+\delta)}(1+\delta)}\\
    &= c (1+\theta_6),
\end{align*}
where we used \cref{th:squareRootOnePlusTheta} and
$$\delta_1'' := (\re{f}^2 \delta_1 + \im{f}^2 \delta_1') / (\re{f}^2+ \im{f}^2),$$ $$\delta_2'' := (\re{g}^2 \delta_2 + \im{g}^2 \delta_2') / (\re{g}^2+ \im{g}^2),$$ $$\theta_5' := \frac{\theta_5}{1+\frac{|f|^2}{|g|^2}}.$$
For $r$, LAPACK 3.9 computes
\begin{align*}
    \float{r} &= f  \sqrt{(1+\frac{g_2}{f_2}(1+\delta))(1+\delta)}(1+\delta)(1+\delta)\\
    &= r \sqrt{1+\theta_6'} (1+\theta_2)\\
    &= r (1+\theta_6''),
\end{align*}
and, for $s$, it computes
\begin{align*}
    h_2 &= (f_2+g_2)(1+\delta) = f_2\left(1+\frac{g_2}{f_2}\right)(1+\delta)\\
    \float{s} &= \conj{g} \frac{\float{r}}{h_2} (1 + \theta_3'')\\
    &= \conj{g} \frac{f  \sqrt{(1+\frac{g_2}{f_2})(1+\delta')(1+\delta)}(1+\delta)(1+\delta)}{f_2\left(1+\frac{g_2}{f_2}\right)(1+\delta)} (1 + \theta_3'')\\
    &= \conj{g} f \sqrt{\frac{(1+\delta')(1+\delta)}{f_2 h_2}} (1 + \theta_6''')\\
    &= \conj{g} f \sqrt{\frac{(1+\theta_6'''')}{|f|^2 (|f|^2+|g|^2)}} (1 + \theta_6''')\\
    &= s (1 + \theta_{10}).
\end{align*}

\subsection{LAPACK 3.10}

When no scaling is necessary and both $f$ and $g$ are non zero, 
LAPACK 3.10 \href{https://github.com/Reference-LAPACK/lapack/blob/lapack-3.10/SRC/cslartg.f90}{\texttt{clartg}} computes the Givens rotation as follows:
\begin{lstlisting}[language=Fortran, firstline=181, lastline=192, firstnumber=181]
         f2 = ABSSQ( f )
         g2 = ABSSQ( g )
         h2 = f2 + g2
         if( f2 > rtmin .and. h2 < rtmax ) then
            d = sqrt( f2*h2 )
         else
            d = sqrt( f2 )*sqrt( h2 )
         end if
         p = 1 / d
         c = f2*p
         s = conjg( g )*( f*p )
         r = f*( h2*p )
\end{lstlisting}
The algorithm of LAPACK 3.10 uses a different strategy from LAPACK 3.9 to compute $c$, $s$ and $r$. First, it computes $p$ as follows:
\begin{align*}
    p &= \frac{1}{\sqrt{f_2 h_2(1+\delta)}(1+\delta)}(1+\delta)\\
    p &= \frac{1}{\sqrt{|f|^2 (|f|^2+|g|^2)(1+\theta_5)(1+\delta)}(1+\delta)}(1+\delta)\\
    &= \frac{1}{|f|\sqrt{|f|^2+|g|^2}}(1+\theta_6'),
\end{align*}
Then $c$ is computed as
\begin{align*}
    \float{c} &= (f_2 p)(1+\delta) \\
    &= \frac{f_2}{\sqrt{f_2 h_2 (1+\delta)}(1+\delta)}(1+\delta)(1+\delta) \\
    &= \frac{(1+\delta)(1+\delta)}{\sqrt{\left(1+\frac{g_2}{f_2}\right)(1+\delta)(1+\delta)}(1+\delta)} \\
    &= c (1+\theta_7),
\end{align*}
where we use the same arguments from \cref{sec:analysis391} for the last step.
Then, $r$ is computed as
\begin{align*}
    \float{r} &= f ( h_2 p ) (1+\theta_2)\\
    &= f \left( \frac{h_2}{\sqrt{f_2 h_2 (1+\delta)}(1+\delta)}(1+\delta) \right) (1+\theta_2)\\
    &= f \sqrt{\frac{1+(g_2/f_2)(1+\delta)}{1+\delta}} (1+\theta_4)\\
    &= r \sqrt{1+\theta_6} (1+\theta_4)\\
    &= r (1+\theta_8),
\end{align*}
and $s$ is computed as follows
\begin{align*}
	\float{s} &= \conj{g} ( f p ) (1+\theta_3') = s (1+\theta_9')\,.
\end{align*}

When $f_2 h_2$ would cause an over- or underflow, the algorithm computes $p = 1 / ( \sqrt{f_2}\sqrt{h_2} )$ instead of $p = 1 / ( \sqrt{f_2 h_2} )$, which increases the accumulation error. We will not give details on this case because it would involve repeating most of the steps from before.
However, we show a numerical experiment that stresses this loss of accuracy in \cref{sec:accuracyTests}.

\subsubsection{New algorithm}

We propose a new algorithm that leads to smaller errors in the worst-case scenario. Here is its unscaled part:
\begin{lstlisting}[language=c++, firstline=321, lastline=350, firstnumber=321]
// Use unscaled algorithm
f2 = abssq(f);
g2 = abssq(g);
h2 = f2 + g2;
// safmin <= f2 <= h2 <= safmax
if( safmin * h2 <= f2 ) {
    // 1 >= f2/h2 >= safmin, and h2/f2 is finite
    c = sqrt( f2 / h2 ); // gamma_4
    r = f / c; // gamma_5
    rtmax = rtmax * 2;
    s = (f2 > rtmin && h2 < rtmax) // Possible improvement: if( f2 > rtmin ) { if( h2 < rtmax ) gamma_7 else gamma_8 } else { if( f2*h2 > safmin ) gamma_7 else gamma_8 }
        ? conj(g) * (f / sqrt(f2 * h2)) // gamma_7
        : conj(g) * (r / h2); // gamma_8
} else {
    /* f2/h2 <= safmin may be subnormal, and h2/f2 may overflow.
        Moreover,
        safmin <= f2*f2 * safmax < f2 * h2 < h2*h2 * safmin <= safmax
        sqrt(safmin) <= sqrt(f2 * h2) <= sqrt(safmax)
        Also,
        g2 >> f2, which means h2 = g2
        and, if f2 / sqrt(f2 * h2) < safmin, then
        h2 / sqrt(f2 * h2) <= (h2*safmin) / f2 <= 1/safmin == safmax
    */
    d = sqrt(f2 * h2); // gamma_{3.5}
    c = f2 / d; // gamma_{4.5}
    r = (c >= safmin)
        ? f / c            // gamma_{5.5}
        : f * ( h2 / d );  // gamma_{5.5}
    s = conj(g) * (f / d); // gamma_{6.5}
}
\end{lstlisting}
Let us analyze each of the terms for this new algorithm.

\begin{enumerate}
\item Common part assumes $\texttt{safmin}\; f_2 \le h_2 \texttt{safmax}$.
\begin{align*}
    f_2 &= (\re{f}^2(1+\delta_1) + \im{f}^2(1+\delta_1'))(1+\delta_3),
    = |f|^2 (1+\delta_1'')(1+\delta_3)\\
    g_2 &= (\re{g}^2(1+\delta_2) + \im{g}^2(1+\delta_2'))(1+\delta_4)
    = |g|^2 (1+\delta_2'')(1+\delta_4).
\end{align*}
\item If $\texttt{safmin}\; h_2 \le f_2$, then $\texttt{safmin} \le f_2/h_2 \le 1$.
\begin{align*}
    \float{c} &= \sqrt{\frac{f_2}{(f_2+g_2)(1+\delta)}(1+\delta)}(1+\delta)\\
    &= \sqrt{\frac{1}{\left(1+\frac{g_2}{f_2}\right)(1+\delta)}(1+\delta)}(1+\delta)\\
    &= \sqrt{\frac{(1+\theta_4)}{\left(1+\frac{|g|^2}{|f|^2}\right)(1+\delta)}(1+\delta)}(1+\delta)
    = c(1+\theta_5)\\
    \float{r} &= \frac{f}{\float{c}}(1+\delta) = r(1+\theta_6)
\end{align*}
Then, if $f_2 > \texttt{rtmin}$ and $h_2 < \texttt{rtmax}$, $\float{s}$ is
\begin{align*}
    \float{s} &= \conj{g} \frac{f}{\sqrt{f_2 h_2 (1+\delta)}} (1 + \theta_4)\\
     &= \conj{g} \frac{f}{\sqrt{|f|^2 (|f|^2+|g|^2) (1+\theta_6)}} (1 + \theta_4) = s (1+\theta_8).
\end{align*}
If not,
\begin{align*}
    \float{s} &= \conj{g} \frac{\float{r}}{h_2} (1 + \theta_3)\\
     &= \conj{g} \frac{f}{h_2 \float{c}} (1 + \theta_3)(1+\delta)\\
     &= \conj{g} \frac{f}{h_2 \sqrt{\frac{f_2}{h_2}(1+\delta)}(1+\delta)} (1 + \theta_3)(1+\delta)\\
     &= \conj{g} \frac{f}{\sqrt{f_2 (f_2+g_2)(1+\delta')(1+\delta)}(1+\delta)} (1 + \theta_3)(1+\delta)\\
     &= \conj{g} \frac{f}{\sqrt{|f|^2 (|f|^2+|g|^2)(1+\theta_6)}(1+\delta)} (1 + \theta_3)(1+\delta) = s (1+\theta_9).
\end{align*}
\item If $\texttt{safmin}\; h_2 > f_2$, then $f_2/h_2 < \texttt{safmin}$ may be subnormal, and $h_2/f_2$ may overflow. Moreover,\\
$\texttt{safmin} \le f_2^2 \; \texttt{safmax} < f_2 h_2 < h_2^2 \; \texttt{safmin} \le \texttt{safmax}$, and then \\
$\sqrt{\texttt{safmin}} \le \sqrt{f_2 h_2} \le \sqrt{\texttt{safmin}}$.
Also, $g_2 \gg f_2$, which means $h_2 = g_2$.
\begin{align*}
	d &= \sqrt{f_2 h_2(1+\delta)}(1+\delta) = \sqrt{f_2 g_2(1+\delta)}(1+\delta) = \sqrt{|f|^2|g|^2}(1+\theta_{4})\\
    \float{c} &= \frac{f_2}{d}(1+\delta)
    = \frac{f_2}{\sqrt{f_2 g_2(1+\delta)}(1+\delta)}(1+\delta)\\
    &= \frac{1}{\sqrt{\frac{g_2}{f_2}(1+\delta)}(1+\delta)}(1+\delta)
    = c(1+\theta_{5})
\end{align*}
Then, if $c > \texttt{safmin}$, $\float{r}$ is
\begin{align*}
    \float{r} &= \frac{f}{\float{c}}(1+\delta) = r (1+\theta_{6})
\end{align*}
If not, then $f_2 / \sqrt{f_2 h_2} < \texttt{safmin}$, which means \\
$h_2 / \sqrt{f_2 h_2} \le (h_2\;\texttt{safmin}) / f_2 \le 1/\texttt{safmin} = \texttt{safmax}$, and then we can compute
\begin{align*}
    \float{r} &= f\frac{h_2}{d}(1+\theta_2)
    = f\frac{g_2}{\sqrt{f_2 g_2(1+\delta)}(1+\delta)}(1+\theta_2)\\
    &= f\frac{1}{\sqrt{\frac{f_2}{g_2}(1+\delta)}(1+\delta)}(1+\theta_2)
    = r(1+\theta_{6})
\end{align*}
And, finally,
\begin{align*}
    \float{s} &= \conj{g} \frac{f}{d} (1 + \theta_3')
	= s (1+\theta_{7}).
\end{align*}
\end{enumerate}
See \cref{tab:comparisonWorstCaseMethods} for a comparison between the three \texttt{clartg} algorithms.

\begin{table}[!htbp]
\centering
\caption{Comparison of the expected errors in the worst-case scenario on each \texttt{clartg} algorithm for $\texttt{safmin}\; h_2 \le f_2$, $\texttt{rtmin} < f_2$ and $h_2 < \texttt{rtmax}$.}
\label{tab:comparisonWorstCaseMethods}
\begin{tabular}{c|c|c|c}
Algorithm: & LAPACK 3.9 & LAPACK 3.10 & Proposed\\\hline
$|c-\float{c}|\,/\,|c|$ & $\gamma_{6}$  & $\gamma_7$ & $\gamma_{5}$ \\
$|r-\float{r}|\,/\,|r|$ & $\gamma_{6}$  & $\gamma_8$ & $\gamma_{6}$ \\
$|s-\float{s}|\,/\,|s|$ & $\gamma_{10}$ & $\gamma_9$ & $\gamma_{8}$ \\
$|\sigma-\float{\sigma}|\,/\,|\sigma|$ & $\gamma_{10}$ & $\gamma_9$ & $\gamma_{8}$ \\
Backward error & $\gamma_{16}$ & $\gamma_{17}$ & $\gamma_{14}$
\end{tabular}
\end{table}

We shall mention that the unscaled part of the new algorithm computes at most 5 floating-point divisions and at most 2 square roots. The unscaled part of the algorithm from LAPACK 3.9 computes at most 7 floating-point divisions, due to the use of \texttt{lapy2}, and 1 square root.
In LAPACK 3.10, the algorithm executes 1 floating-point division and 1 square root at most.

\section{Numerical results}
\label{sec:numerical}

The tests in this section compare the accuracy and timing between different algorithms for generating Givens rotations. For accuracy, we use a \texttt{C++} version of each algorithm. For the timing experiments, we use versions in Fortran.

For the accuracy tests, all errors are measured in double precision.
The outputs $c$, $s$ and $r$ of the proposed algorithm, in double precision, are used as the correct answer for the rotation. It makes no difference which double-precision algorithm we choose for the comparisons that follow.
In the following, "cast from double to float" stands for applying the proposed algorithm, in double-precision, and then casting the output to single precision.

\subsection{Setup}
\label{sec:numericalSetup}

The tests are performed in a Linux \texttt{20.04.3-Ubuntu SMP x86\_64} machine with kernel \texttt{5.11.0-41-generic}.
We use the GCC 9.3.0 compiler with default configurations. 


Random pairs $(f,g)$ are generated using \texttt{rand()} from \texttt{stdlib.h}, and then converted into float using the following routines.
\begin{lstlisting}[language=c++]
inline constexpr double randToRadians( int N ) {
    return ((double) N / RAND_MAX) * (2.*M_PI);
}
\end{lstlisting}
\begin{lstlisting}[language=c++]
inline float randToModulus( int N ) {
    using std::exp2f;
    const int expm = std::numeric_limits<float>::min_exponent;
    const int expM = std::numeric_limits<float>::max_exponent;
    const float safmaxExp = std::min(1-expm,expM-1);
    const int   digits = std::numeric_limits<float>::digits;

    const float rhoMax = (safmaxExp-digits+1) / 2.0F - 1;
    const float rhoMin = -rhoMax;

    return exp2f( rhoMin + (rhoMax-rhoMin) * ((float)N) / RAND_MAX );
}
\end{lstlisting}
These two routines generate the polar coordinates of $f$ and $g$ as shown in the following piece of code
\begin{lstlisting}[language=c++]
// Random pairs (f,g)
{
    double theta = randToRadians( rand() );
    double phi   = randToRadians( rand() );
    float r1 = randToModulus( rand() );
    float r2 = randToModulus( rand() );
    f.real( r1 * ((float) cos(theta)) );
    f.imag( r1 * ((float) sin(theta)) );
    g.real( r2 * ((float) cos(theta+phi)) );
    g.imag( r2 * ((float) sin(theta+phi)) );
}
\end{lstlisting}
The angles and the log of the lengths are approximately uniform distributed. See \cref{fig:distributionInput}. Moreover, this choice of \texttt{rhoMin} and  \texttt{rhoMax} in \texttt{randToModulus} allows us to test only the unscaled part of each algorithm.
We choose 1 for the random generator seed for no particular reason.

\begin{figure}[!htbp]
\centering
\includegraphics[width=.5\linewidth]{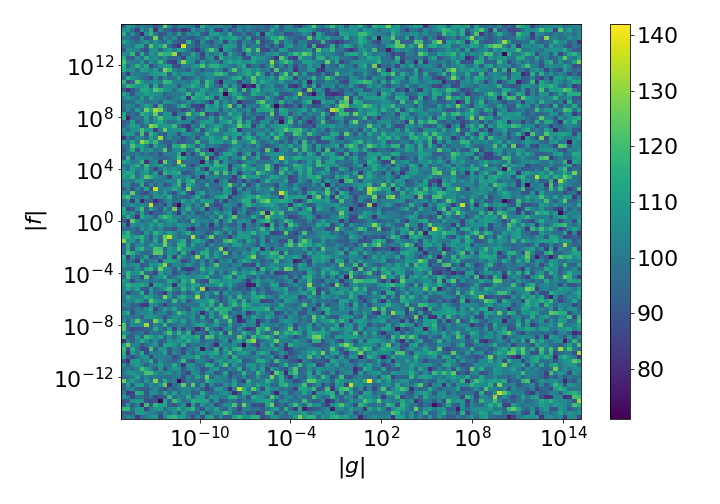}\\
\includegraphics[width=\linewidth]{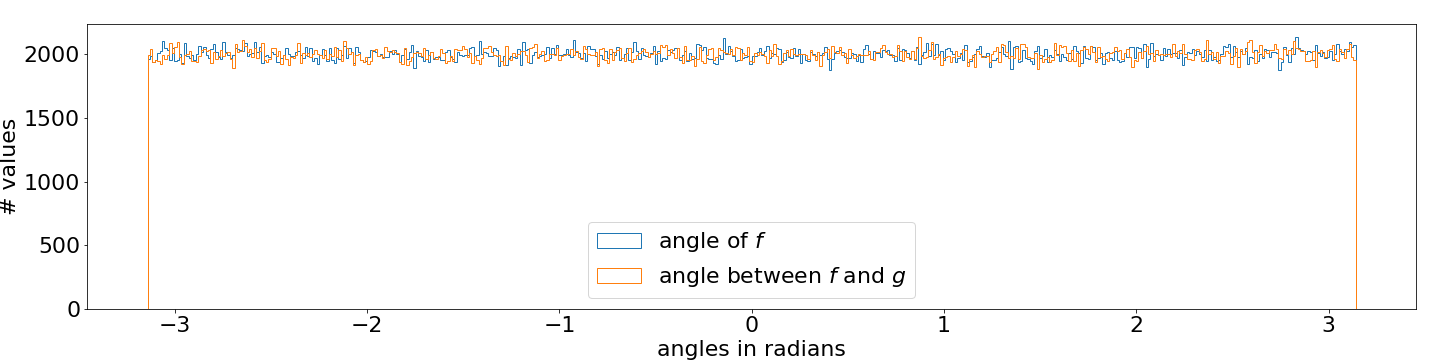}
\caption{Distribution of the lengths and angles of the input data.}
\label{fig:distributionInput}
\end{figure}

\subsection{Accuracy on a single rotation}
\label{sec:accuracyTests}

\Cref{fig:singValues} and \cref{tab:singValues} show the error in the singular values, and \cref{fig:backError} and \cref{tab:backError} show the relative backward errors $\|\float{Q}^H(\float{r},0)^T - (f,g)^T\|_2 / \|(f,g)\|_2$.
We run each code several ($10^6$) times with random input data.
As expected, applying the double precision algorithm and then casting the solution to single precision is at least as accurate as trying to compute the rotation in single precision.
As the theory predicts (see \cref{tab:comparisonWorstCaseMethods}), the new proposed algorithm is more accurate than the algorithms from LAPACK 3.9 and LAPACK 3.10.

\begin{figure}[!htbp]
\includegraphics[width=\linewidth]{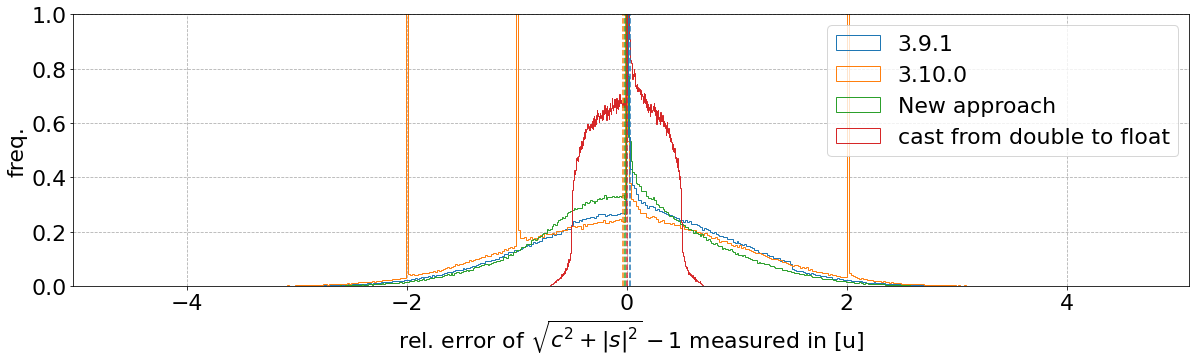}
\caption{Histogram of the error in the singular values, $Err := \sqrt{c^2+|s|^2} - 1$, for each algorithm measured in unit roundoff.}
\label{fig:singValues}
\end{figure}

\begin{table}[!htbp]
\caption{Error in the singular values $Err := \sqrt{c^2+|s|^2} - 1$ for each algorithm measured in unit roundoff $u$.}
\label{tab:singValues}
\centering
\begin{tabular}{r|c|c|c|c|c}
& avg(Err) & std(Err) & avg(|Err|) & std(|Err|) & max(|Err|)\\\hline
3.9 & 3.29e-02 & 7.08e-01 & 4.45e-01 & 5.52e-01 & 4.38e+00\\\hline
3.10 & -3.25e-02 & 9.47e-01 & 6.64e-01 & 6.76e-01 & 4.66e+00\\\hline
New & -1.14e-02 & 6.34e-01 & 3.91e-01 & 5.00e-01 & 3.94e+00\\\hline
Cast & 2.22e-03 & 2.23e-01 & 1.50e-01 & 1.65e-01 & 7.82e-01
\end{tabular}
\end{table}

\begin{figure}[!htbp]
\includegraphics[width=\linewidth]{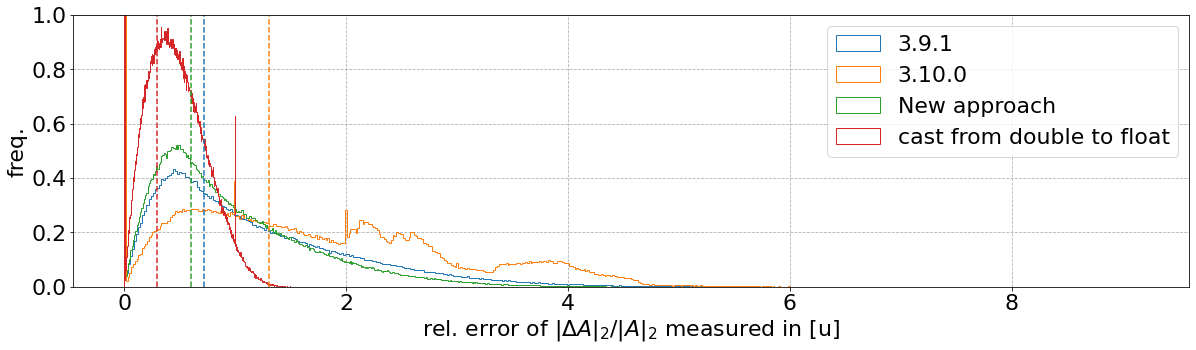}
\caption{Histogram of the relative backward errors $\|\float{Q}^H(\float{r},0)^T - (f,g)^T\|_2 / \|(f,g)\|_2$ for each algorithm measured in unit roundoff.}
\label{fig:backError}
\end{figure}

\begin{table}[!htbp]
\caption{Relative backward errors $\|\float{Q}^H(\float{r},0)^T - (f,g)^T\|_2 / \|(f,g)\|_2$ for each algorithm measured in unit roundoff $u$.}
\label{tab:backError}
\centering
\begin{tabular}{r|c|c|c|c|c}
& avg(Err) & std(Err) & avg(|Err|) & std(|Err|) & max(|Err|)\\\hline
3.9 & 7.18e-01 & 9.08e-01 & 7.18e-01 & 9.08e-01 & 9.15e+00 \\\hline
3.10 & 1.31e+00 & 1.27e+00 & 1.31e+00 & 1.27e+00 & 7.96e+00 \\\hline
New & 6.05e-01 & 7.35e-01 & 6.05e-01 & 7.35e-01 & 5.56e+00 \\\hline
Cast & 2.95e-01 & 3.09e-01 & 2.95e-01 & 3.09e-01 & 1.59e+00
\end{tabular}
\end{table}

In \cref{fig:funcSigma}, we analyze the graph of $(|f|,|g|) \to \sqrt{c^2+|s|^2}-1$.
The profile of errors in LAPACK 3.9, the proposed and the "cast from double to float" algorithms are similar in two aspects:
	(1) small errors ( $\le$ 10\% of $u$ ) when $|f| \gg |g|$;
	(2) several tiny regions with $\sigma > 1$ and $\sigma < 1$ that do not appear to follow any pattern.
When $|f| \gg |g|$, $c \approx 1$, $r \approx f$ and $s \approx \conj{g} / |f|$, and this case is approximated very accurately by the algorithm in LAPACK 3.9 and the proposed algorithm.
In LAPACK 3.10, however, if $|f|^2 \ll 1$ or $|f|^2 + |g|^2 \gg 1$, more accumulation error is introduced, and this includes the case where $|f| \gg |g|$. That is why we observe two regions of low accuracy on the region $|f| > |g|$.

\begin{figure}[!htbp]
\includegraphics[width=.49\linewidth]{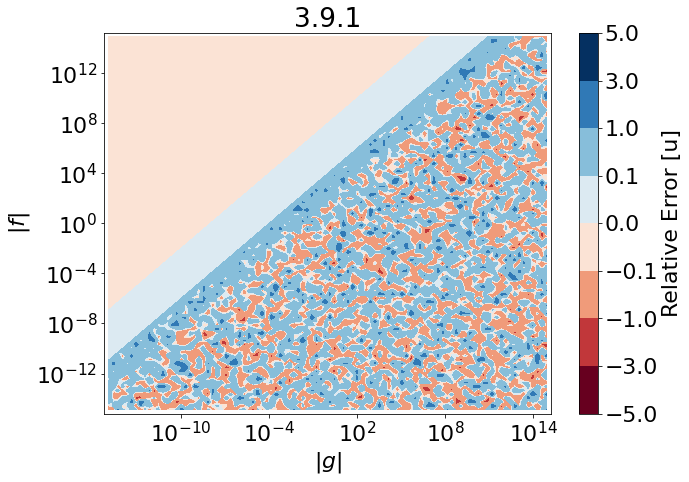}\hfill
\includegraphics[width=.49\linewidth]{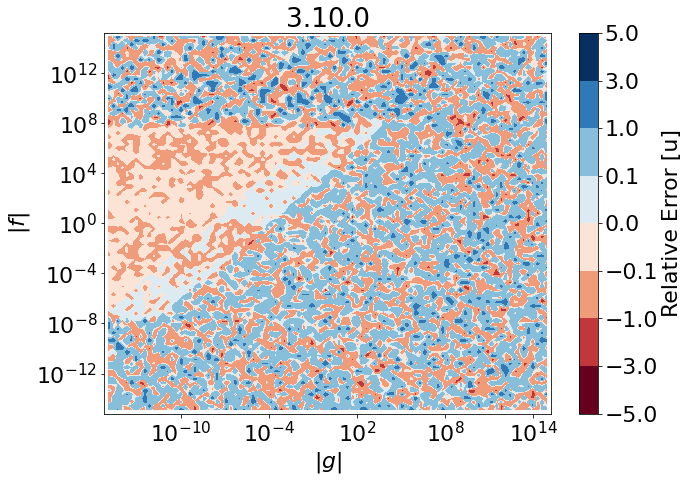}\\
\includegraphics[width=.49\linewidth]{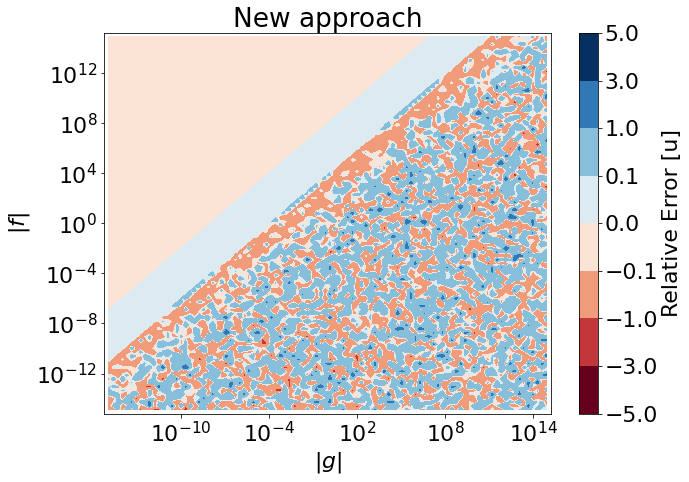}\hfill
\includegraphics[width=.49\linewidth]{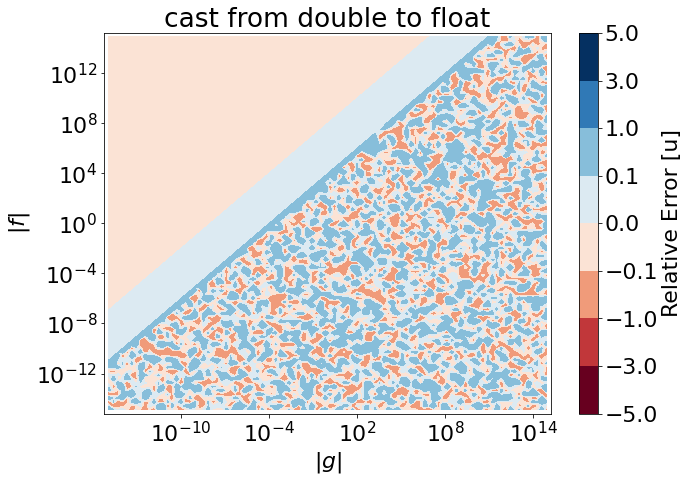}
\caption{Graph of $(|f|,|g|) \to \sqrt{c^2+|s|^2}-1$ for each \texttt{clartg} algorithm.}
\label{fig:funcSigma}
\end{figure}


\subsection{Accuracy of multiple rotations on 2-by-2 matrices}
\label{sec:accuracyMat2x2}

The loss of accuracy of a single applied rotation can be harmless to the overall numerical computation. However, applying multiple rotations to a matrix may deteriorate the expected final result. In this section, we apply several rotations to an initial unitary matrix and (1) predict the norm of the final matrix, and (2) show the loss of orthogonality.
We use double precision to compute the matrix-matrix multiplications and for measuring the errors.

Let $X$ be a random variable associated with the distribution of the singular values of a \texttt{clartg} algorithm, and define $Y := X^M$ for some big number $M$.
We can estimate $\mu_Y := E[Y]$ and $\sigma_Y := \sqrt{Var[Y]}$ from $\mu_X := E[X]$ and $\sigma_X := \sqrt{Var[X]}$. In \cref{sec:statisticsProdNumbers}, we show that $Y$ can be approximated by a log-normal distribution with
\begin{align}\label{eq:estimateY}
\mu_Y 
\approx \mu_X^M\,,\qquad
\sigma_Y
\approx \mu_X^M \sqrt{ e^{M\left(\frac{\sigma_X}{\mu_X}\right)^2}-1 }\,.
\end{align}

\Cref{fig:estimateProdSingVal,tab:estimateProdSingVal} compare the estimates above using data from \cref{tab:singValues} with experimental measurements. We choose $M = 10^5$. To generate the experimental data, we multiply the singular values of $M$ rotation matrix and repeat this procedure $N = 10^3$ times. We generate the $M N = 10^8$ input pairs $(f,g)$ using the procedure described in \cref{sec:numericalSetup}.
The curves predicted are very accurate, and \cref{tab:estimateProdSingVal} shows that the prediction error is less than 7\% for both average and standard deviation values.
These curves help predict the norm of a given matrix in $\Complex^{2\times 2}$ after $M$ rotations as explained in \cref{sec:accuracyComplex}.
So, $\left\| \prod_{i=1}^{m}\float{Q_i} \right\|_F = \left(\prod_{i=1}^{m}\sigma_i\right) \|\tilde{Q}\|_F = \sqrt{2} \left(\prod_{i=1}^{m}\sigma_i\right)$, where $\tilde{Q} \in \Complex^{2\times 2}$ is unitary.
This is exactly what we observe in practice.
Finally, observe that the new approach is better than both LAPACK 3.9 and LAPACK 3.10 algorithms, which is a natural extension from what was observed in \cref{sec:accuracyTests}.

\begin{figure}[!htbp]
\includegraphics[width=\linewidth]{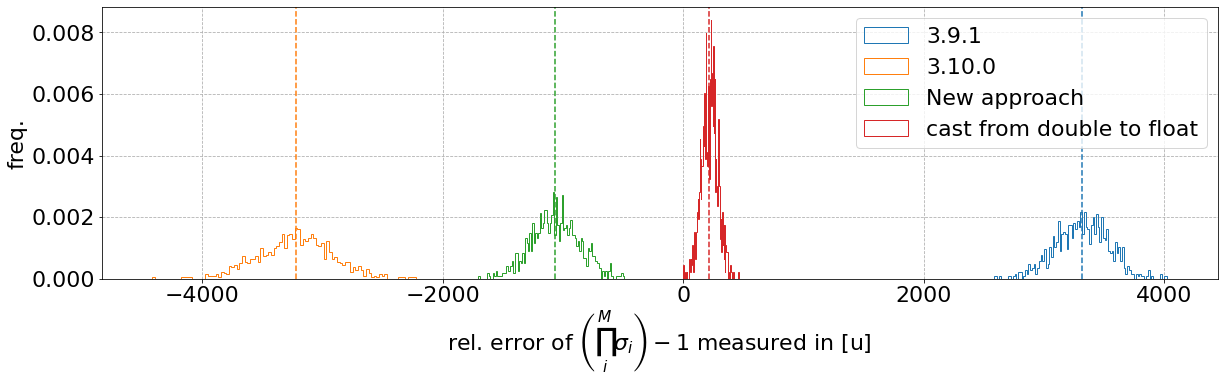}\\
\includegraphics[width=\linewidth]{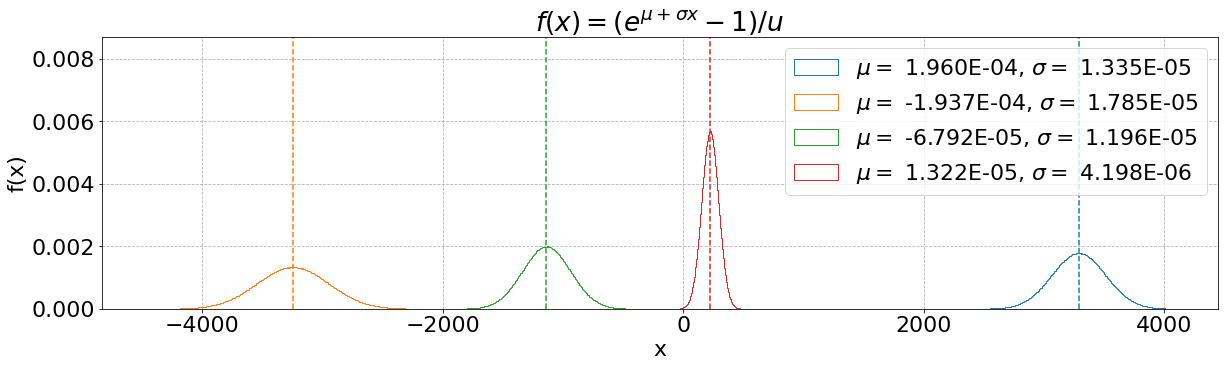}
\caption{Top: Histogram of $\left(\prod_{i=1}^M \sigma_i\right) - 1$, $M = 10^5$, for each algorithm measured in unit roundoff. Bottom: Estimates using \cref{eq:estimateY} and data from \cref{tab:singValues}.}
\label{fig:estimateProdSingVal}
\end{figure}

\begin{table}[!htbp]
\caption{Error in the singular values $\left(\prod_{i=1}^M \sigma_i\right) - 1$, $M = 10^5$, for each algorithm measured in unit roundoff $u$. \Cref{eq:estimateY} and data from \cref{tab:singValues} are used to compute $\mu_Y$ and $\sigma_Y$.}
\label{tab:estimateProdSingVal}
\centering
\begin{tabular}{r|c|c|c|c}
& avg(Err) & $(\mu_Y - 1)/u$ & std(Err) & $(\sigma_Y - 1)/u$\\\hline
3.9 & 3.31e+03 & 3.29e+03 & 2.18e+02 & 2.24e+02 \\\hline
3.10 & -3.22e+03 & -3.25e+03 & 3.06e+02 & 2.99e+02 \\\hline
New & -1.07e+03 & -1.14e+03 & 2.00e+02 & 2.01e+02 \\\hline
Cast & 2.17e+02 & 2.22e+02 & 6.83e+01 & 7.04e+01
\end{tabular}
\end{table}

\subsection{Accuracy of multiple rotations on 3-by-3 matrices}

Now, we want to measure the accumulation error when rotating $M$ times a unitary matrix  $V_0 \in \Complex^{3 \times 3}$. For each rotation $\float{Q}_i \in \Complex^{2 \times 2}$, we define $\float{Q}_{i,IJ} \in \Complex^{3 \times 3}$ as the rotation $\float{Q}_i$ in the coordinate directions $(I,J)$. Then, let 
\begin{align*}
	V_M := Q_{M,I_MJ_M} (Q_{M-1,I_{M-1}J_{M-1}} \cdots (Q_{2,I_2J_2}(Q_{1,I_1J_1} V_0))
\end{align*}
where
\begin{align*}
	(I_k,J_k) = \begin{cases}
		(1,2) & \text{if } (k-1) \mod 3 = 0,\\
		(2,3) & \text{if } (k-2) \mod 3 = 0,\\
		(1,3) & \text{if } k \mod 3 = 0.
	\end{cases}
\end{align*}
We use the input data from \cref{sec:accuracyMat2x2}.

Since the rotation is applied on different rows at each time, the columns of the final matrix $V_M$ are not orthogonal, which differs from the $\Complex^{2 \times 2}$ case.
The orthogonality of the columns of $V_M$ is measured as follows:
(1) compute $S := V_M^H V_M$;
(2) compute the average of the absolute values of the off-diagonal elements of $S$.
We are still able to estimate the norm of $V_M$ using $\prod_{i=1}^M \sigma_i$.
Observe that each row of $V_0$ is rotated $(2/3) M$ times, so the norm of each row of $V_M$ is roughly equal to $\left(\prod_{i=1}^M \sigma_i\right)^{\frac23}$.
Therefore, $\| V_M \|_F \approx \sqrt{3} \left(\prod_{i=1}^M \sigma_i\right)^{\frac23}$.
Since $\prod_{i=1}^M \sigma_i$ is close to 1, we can avoid the nonlinearity by using the Taylor series centered at one:
\begin{align}\label{fig:estimateNormMat3x3}
	\frac{\left\| V_M \right\|_F}{\sqrt{3}} \approx 1 + \frac23 \left(\prod_{i=1}^{M}\sigma_i -1 \right)
	\quad \Leftrightarrow \quad
	\frac32\left( \frac{\left\| V_M \right\|_F}{\sqrt{3}} - 1 \right) \approx \prod_{i=1}^M\sigma_i - 1.
\end{align}

\begin{figure}[!htbp]
\includegraphics[width=\linewidth]{prodSingValue}\\
\includegraphics[width=\linewidth]{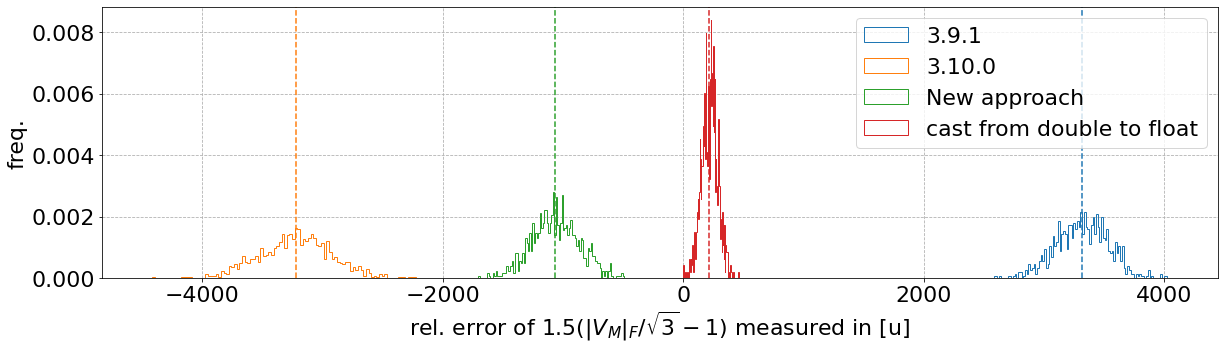}\\
\includegraphics[width=\linewidth]{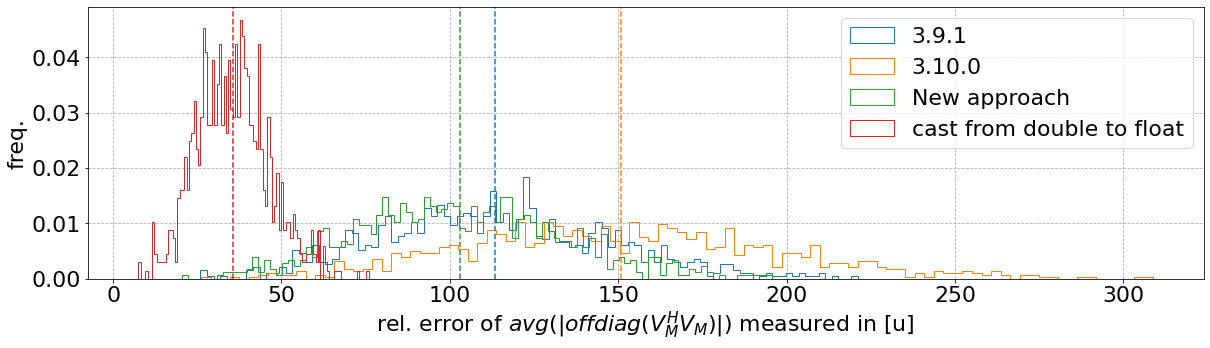}
\caption{Histograms of $\left(\prod_{i=1}^M \sigma_i\right) - 1$, $1.5(\|V_M\|_F/\sqrt{3}-1)$, and $\text{avg}(|\text{offdiag}(V_M^H V_M)|)$, $M = 10^5$, for each algorithm measured in unit roundoff.}
\label{fig:estimateProdMat3x3}
\end{figure}

\Cref{fig:estimateProdMat3x3} shows the experimental results. Notice that \cref{fig:estimateNormMat3x3} is indeed a good approximation for this dataset.
Observe that the new approach is better than both LAPACK 3.9 and LAPACK 3.10 algorithms also when it comes to preserving the orthogonality of the columns or rows after $M$ rotations.

The analysis for the rotation of $3\times 3$ matrices naturally extends to general $n\times n$ matrices, $n > 3$.

\subsection{Performance tests}

The last set of tests measures the time of the different algorithms for computing  Givens rotations. For that, we use Fortran implementations in the LAPACK library.
We compile LAPACK using the release flags \texttt{-O2 -DNDEBUG}, and the following commits:
\begin{itemize}
\item LAPACK 3.9: \href{https://github.com/Reference-LAPACK/lapack/tree/lapack-3.9}{github.com/Reference-LAPACK/lapack/tree/lapack-3.9}
\item LAPACK 3.10: \href{https://github.com/Reference-LAPACK/lapack/tree/lapack-3.10}{github.com/Reference-LAPACK/lapack/tree/lapack-3.10}
\item New: \href{https://github.com/Reference-LAPACK/lapack/pull/631/commits/c362fff1eee80fdd88b6fd60be3fb9d045cb08bb}{github.com/Reference-LAPACK/lapack/pull/631/commits/c362fff}
\end{itemize}

We generate $(f,g)$ as explained in \cref{sec:numericalSetup}. To better cover the different input configurations, we test several scenarios of \texttt{(rhoMin,rhoMax)} as shown in \cref{tab:timingRange}.
We run the same code with \num{1e7} different pairs $(f,g)$ to obtain each average value.
Moreover, we run 3 times each test and take the lower execution time to reduce the interference of other processes running in the machine.

\begin{table}[!htbp]
\caption{Average times in nanoseconds using multiple values $(\rho_{min},\rho_{max})$ to generate $(f,g)$.}
\label{tab:timingRange}
\centering
\begin{tabular}{lccrrrrrr}
\toprule
& $f$ & $g$ & \multicolumn{3}{c}{\texttt{clartg}} & \multicolumn{3}{c}{\texttt{zlartg} + cast} \\
& $(\rho_{min},\rho_{max})$ & $(\rho_{min},\rho_{max})$ &  3.9 &  3.10 &  New &  3.9 &  3.10 &  New \\
\midrule
1 &               (-50.5, 50.5) &               (-50.5, 50.5) &   61.0 &    38.7 & 43.7 &   89.7 &    61.7 & 70.1 \\
2 &                   (-63, 62) &                   (-63, 62) &   68.8 &    45.3 & 45.3 &   89.3 &    61.7 & 70.1 \\
3 &                (-63, -50.5) &                  (50.5, 62) &  120.5 &    53.4 & 38.9 &   90.8 &    61.7 & 81.8 \\
4 &                  (50.5, 62) &                (-63, -50.5) &   78.3 &    54.2 & 40.4 &   90.4 &    61.7 & 67.7 \\
5 &                 (-125, 127) &                 (-125, 127) &   93.6 &    57.5 & 64.0 &   89.4 &    61.6 & 69.9 \\
6 &                 (-125, -63) &                   (62, 127) &  113.1 &    57.7 & 65.8 &   90.0 &    61.6 & 81.2 \\
7 &                   (62, 127) &                 (-125, -63) &   80.4 &    52.5 & 63.9 &   89.8 &    61.8 & 66.7 \\
\bottomrule
\end{tabular}
\end{table}

In the first scenario of \cref{tab:timingRange}, only the unscaled part of each \texttt{clartg} algorithm is used, and that is why all versions of \texttt{clartg} spend less time to finish.
The second, third and fourth scenarios stress the scaled part of the \texttt{clartg} algorithms from LAPACK 3.9 and LAPACK 3.10, but still use only the unscaled part of the new approach.
The remaining scenarios stress the scaled part of each \texttt{clartg} algorithm.
All these scenarios stress only the unscaled part of the double precision algorithms, \texttt{zlartg}.

We observed lower execution times for the algorithms from LAPACK 3.10 on cases 1 and 5 -- 7. The algorithms from LAPACK 3.9 have the highest execution times in all scenarios. Mind that the Givens rotations in 3.10 use fewer divisions than the other algorithms and that the former was designed to be computationally efficient.
In cases 2 -- 4, the new \texttt{clartg} algorithm uses only its unscaled part, so it is supposed to be faster than the algorithm from LAPACK 3.10.
When there is a scaling in \texttt{clartg}, we observe that single and double precision algorithms have close execution times, and sometimes the double-precision algorithm is faster. We shall highlight that the double-precision algorithm is always at least as precise as the single-precision one.

\section{Conclusions}
\label{sec:concl}

In this document, we analyzed different algorithms for generating Givens rotations and compared them via both theoretical worst-case scenarios and numerical experiments.

We briefly discussed the differences between the real-valued algorithms in LAPACK 3.9 and LAPACK 3.10 and concluded that they approximate the same quantities with different choices of signs. The choice of signs in the algorithm from LAPACK 3.10 is more adequate since it matches real- and complex-arithmetic outputs.

We analyzed the \texttt{clartg} algorithms in LAPACK 3.9 and LAPACK 3.10 and, after that, proposed a new and more accurate algorithm. We provided several numerical experiments that validate the theory. We also showed that the bias in the error of the rotation singular values is less biased, i.e., closer to zero. We couldn't, however, arrive at the bias of the ``cast from double to float'' algorithm. So, we believe there are still opportunities to improve it. The lower the bias, the more accurate the application of multiple rotations is.

The proposed \texttt{clartg} algorithm is slower than the algorithm in LAPACK 3.10 in most of the cases, which is expected due to the additional floating-point divisions and, possibly, additional square roots. We believe that the best algorithm for generating rotation matrices should be the most accurate, especially when the same rotation is applied several times. It is worth mentioning that we verified that the most accurate strategy to generate Givens rotations is to use an algorithm in high precision and then cast the output to the desired lower precision.

Given a \texttt{clartg} algorithm, we can find the expected value and variance in the singular values of its output rotation matrices. Then, we use those two values to estimate the distribution of the product of $M$ singular values, when $M$ is big.
In this work, we successfully use this estimated distribution to predict the norm of a matrix rotated $M$ times. One may use this information, for example, to rescale without having to compute the norm of the final matrix.

The discussion on how to bound the square root of rounding errors is a byproduct of the numerical analysis presented in this document. Those results are particularly interesting when there are a few errors accumulated, which is the case in the analysis of a single Givens rotation. We verify that $\sqrt{1+\theta_n}$ is bounded by $\gamma_{\lfloor \frac{n}2 \rfloor + 1}$. In practice, one may expect $1+\gamma_{\frac{n}2}$ for small $n$.

\appendix

\section{Expectation and Variance of products}
\label{sec:statisticsProdNumbers}

Let $m$ be a big number.
Suppose $\{X_i\}_{i=1}^m$ is a set of random variables from a distribution with expectation $\mu_X$ and variance $\sigma_X^2$, and define $Y := \prod_{i=1}^m X_i$. It is possible to estimate $E[Y]$ and $Var[Y]$ from $\mu_X$ and $\sigma_X$ using the distributions $Z_i := \log(X_i)$ and $W := \log(Y)$.
Using the Taylor series, we obtain
\begin{align*}
	E[Z_i] = E[\log(X_i)] &= \sum_{n=0}^{\infty} \frac{\log^{(n)}(\mu_X)}{n!}  E[(X_i-\mu_X)^n]\\
	&= \log(\mu_X) + \sum_{n=1}^{\infty} \frac{(-1)^{n-1}}{n} \frac{E[(X_i-\mu_X)^n]}{\mu_X^n}.
\end{align*}
Since $E[(X_i-\mu_X)] = 0$ and, by definition, $E[(X_i-\mu_X)^2] = \sigma_X^2$, we may write
\begin{align*}
	E[Z_i] = \log(\mu_X) - \frac12\left(\frac{\sigma_X}{\mu_X}\right)^2 + \sum_{n=3}^{\infty} \frac{(-1)^{n-1}}{n} \frac{E[(X_i-\mu_X)^n]}{\mu_X^n}.
\end{align*}
For the variance, using $\log^{(n)}(\mu_X) = (-1)^{n-1} (n-1)! / \mu_X^n$, we obtain
\begin{align*}
	Var[Z_i] &= Var\left[\log'(\mu_X) (X_i-\mu_X)\right] + Var\left[\sum_{n=2}^{\infty}\frac{\log^{(n)}(\mu_X)}{n!} (X_i-\mu_X)^n\right]\\
	&\quad + Cov\left[\log'(\mu_X) (X_i-\mu_X),\sum_{n=2}^{\infty}\frac{\log^{(n)}(\mu_X)}{n!} (X_i-\mu_X)^n\right]
\end{align*}
then, using $\log^{(n)}(\mu_X) = (-1)^{n-1} (n-1)! / \mu_X^n$ and $Cov[A,B] := E[AB] - E[A]E[B]$, we obtain
\begin{align*}
	Var[Z_i] &= \left(\frac{\sigma_X}{\mu_X}\right)^2 + Var\left[\sum_{n=2}^{\infty}\frac{(-1)^{n-1}}{n \mu_X^n} (X_i-\mu_X)^n\right]\\
	&\quad + E\left[\frac{(X_i-\mu_X)}{\mu_X}\sum_{n=2}^{\infty}\frac{(-1)^{n-1}}{n \mu_X^n} (X_i-\mu_X)^n\right]\\
	&= \left(\frac{\sigma_X}{\mu_X}\right)^2 + Var\left[\sum_{n=2}^{\infty}\frac{(-1)^{n-1}}{n \mu_X^n} (X_i-\mu_X)^n\right]\\
	&\quad + \sum_{n=3}^{\infty}\frac{(-1)^n}{n-1} \frac{E\left[(X_i-\mu_X)^n\right]}{\mu_X^n}
\end{align*}
We may truncate the two series to obtain:
$$\mu_Z := E[Z_i] \approx \log(\mu_X) -\frac{1}{2}\left(\frac{\sigma_X}{\mu_X}\right)^2,\quad \sigma_Z := \sqrt{Var[Z_i]} \approx \frac{\sigma_X}{\mu_X}.$$

Now, notice that $\;W = \log(Y) \;\Leftrightarrow\; W = \sum_{i=1}^m Z_i$.
Since $m$ is big, we may apply the 
Central Limit Theorem to conclude that $W$ is approximately a normal distribution with average $m\mu_Z$ and standard deviation $\sqrt{m}\sigma_Z$.
Finally, $Y = \exp(W)$ is a log-normal distribution with
\begin{align*}
\text{E}\left[Y\right] 
&= \exp\left(\mu_W + \frac{\sigma_W^2}{2}\right)\\
&\approx \exp\left(m\mu_Z + m\frac{\sigma_Z^2}{2}\right)\\
&\approx \exp\left(m\log(\mu_X) -\frac{m}{2}\left(\frac{\sigma_X}{\mu_X}\right)^2 + \frac{m}2 \left(\frac{\sigma_X}{\mu_X}\right)^2 \right)\\
&= \mu_X^m
\end{align*}
and
\begin{align*}
\text{Var}\left[Y\right]
&= (\exp(\sigma_W^2)-1) \exp(2\mu_W+\sigma_W^2)\\
&\approx (\exp(m\sigma_Z^2)-1) \exp(2m\mu_Z+m\sigma_Z^2)\\
&\approx \left(\exp\left(m\left(\frac{\sigma_X}{\mu_X}\right)^2\right)-1\right) \exp\left(2m\log(\mu_X)\right)\\
&= (\mu_X^m)^2\left(e^{m\left(\frac{\sigma_X}{\mu_X}\right)^2}-1\right).
\end{align*}

\bibliographystyle{plain}
\bibliography{references}

\end{document}